\documentclass[preprint,12pt]{elsarticle}

\usepackage{a4wide,amsfonts, amsmath, amscd}
\usepackage[psamsfonts]{amssymb}

\usepackage{graphicx}
\usepackage[cp1251]{inputenc}
\usepackage{amssymb}
\usepackage{amsthm}
\usepackage{cmap}

\usepackage{amsmath}

\biboptions{sort&compress}

\newcommand{\sysn}{\left\{\begin{array}{rcl}}
\newcommand{\sysk}{\end{array}\right.}

\newtheorem{theorem}{Theorem}[section]
\newtheorem{lemma}[theorem]{Lemma}

\theoremstyle{example}

\newtheorem{proposition}[theorem]{Proposition}
\theoremstyle{definition}
\newtheorem{definition}[theorem]{Definition}

\newtheorem{corollary}[theorem]{Corollary}

\journal{...}

\begin{document}

\title{Dieudonn\'{e} completeness of function spaces}

\author[affil1]{Mikhail Al'perin}

\address[affil1]{Krasovskii Institute of Mathematics and Mechanics,}

\ead{alper@mail.ru}

\author[affil1,affil2]{Alexander V. Osipov}
\address[affil2]{Ural Federal
 University, Yekaterinburg, Russia}

\ead{OAB@list.ru}

\begin{abstract}

A space is called {\it Dieudonn\'{e} complete} if it is complete
relative to the maximal uniform structure compatible with its
topology.

In this paper, we investigated when the function space $C(X,Y)$ of all
continuous functions from a topological space $X$ into a uniform
space $Y$ with the topology of uniform convergence on a family of
subsets of $X$ is Dieudonn\'{e} complete. Also we proved a generalization of the Eberlein--\v{S}mulian theorem to the class of Banach spaces.
\end{abstract}

\tnotetext[label1]{The research of the second author was supported
by the Russian Science Foundation (RSF Grant No. 23-21-00195).}


\maketitle 


\section{Introduction}

By a space $X$ in this article is understood a Tychonoff
topological space. Recall that a Tychonoff space ($T_1$-space +
completely regular) is a $T_1$-space such that if $F$ is closed
and $x\not\in F$ then there is a continuous function
$f:X\rightarrow [0,1]$ such that $f(F)=0$ and $f(x)=1$. Tychonoff
spaces are also exactly the uniformizable spaces, i.e., $X$ is
Tychonoff if and only if its topology is generated by a
uniformity.

A space is called {\it Dieudonn\'{e} complete} if it is complete relative to the maximal uniform structure compatible with its topology. The Dieudonn\'{e} complete spaces can be characterized as the homeomorphic images of closed subspaces of a product of metric spaces~\cite{Eng}.

A special and, as it turned out, important form of topological
completeness was introduced by E.~Hewitt and L.~Nachbin. Spaces
satisfying this condition of completeness are called realcompact
spaces. Recall that a space is called {\it realcompact} if it is
homeomorphic to a closed subspace of the space $\mathbb{R}^{\tau}$
for a certain $\tau$. Over the course of the years these spaces
have been studied under a variety of different names: {\it
$e$-complete, functionally closed, Hewitt-Nachbin, realcomplete,
replete} and {\it saturated}. Hewitt, who introduced these spaces
to the literature, originally called them $Q$-spaces
\cite{Hewitt}. The term {\it realcompact}, popularized by the
enormous success of Gillman and Jerison's work \cite{GJ}, is now
prevalent in the literature.

If there does not exist an Ulam measurable cardinal, then a space
is realcompact if and only if it is Dieudonn\'{e} complete
\cite{Eng}.

 A study of some convergence properties in function spaces
is an important task of general topology. The general question in
the theory of function spaces is to characterize topological
properties of a space of functions on a topological space $X$.

Note that a space $C_p(X)$ of all real-valued continuous functions
with the pointwise convergence topology is realcompact if and only
if it is Dieudonn\'{e} complete. This follows from the fact that
the Suslin number $c(C_p(X))$ is countable \cite{ArPo}.

In \cite{Os20}, we investigated the realcompactness of the space
$B(X)$ of real-valued Baire functions with the pointwise
convergence topology.

Recall that a subset $A$ of a space $X$ is called a {\it
(topologically) bounded} subset (in $X$) if every continuous
real-valued function on $X$ is bounded on $A$.

A topological space $X$ is called a {\it $\mu$-space} if every
bounded subset of $X$ is relatively compact \cite{Buch}. Every
Dieudonn\'{e} complete space is a $\mu$-space.

\medskip

In the functional analysis, the Eberlein--\v{S}mulian theorem is a
result that relates three different kinds of weak compactness in a
Banach space. In 1940, \v{S}mulian proved that weakly relatively
compact sets in a Banach space $E$ are weakly relatively
sequentially compact. Dieudonn\'{e} and Schwartz extended this
result to locally convex spaces admitting a weaker metrizable
topology. Eberlein showed that weakly relatively countably compact
sets are weakly relatively compact for a Banach space $E$.

\begin{theorem}(Eberlein--\v{S}mulian) If $E$ is a Banach space and $A$ is a subset of $E$, then the following statements are equivalent:

(1) the weak closure of $A$ is weakly compact;

(2) each sequence of elements of $A$ has a subsequence that is weakly convergent in $E$;

(3) each sequence of elements of $A$ has a weak cluster point in $E$.
\end{theorem}

Later, the Eberlein--\v{S}mulian theorem was generalized in several directions. One of them is related to the expansion of the class of topologies in which the Eberlein--\v{S}mulian theorem holds, and is based on the concept of {\it angelic space} introduced by Fremlin \cite{freml}.

\medskip

In this paper, we to study when the function space $C(X,Y)$ of all
continuous functions from a topological space $X$ into a uniform
space $Y$ with the topology of uniform convergence on a family of
subsets of $X$ is Dieudonn\'{e} complete. In the last section we proved a generalization of the Eberlein--\v{S}mulian theorem to the class of Banach spaces.

\section{Notation and terminology}
 The set of positive integers is denoted by $\mathbb{N}$ and
$\omega=\mathbb{N}\cup \{0\}$. Let $\mathbb{R}$ be the real line. We denote by $\overline{A}$ (or $Cl_X A$)
the closure of $A$ (in $X$).

\begin{definition}\label{t1_2_1}{\rm\cite{kell}}
Let $X$ be a topological space, $\lambda\subseteq 2^X$, $(Y,\mu)$
be a uniform space. A topology on $C(X,Y)$ generated by the
uniformity
$$\nu=\{\langle A,M\rangle \subseteq C(X,Y)\times
C(X,Y):A\in\lambda,M\in\mu\}$$ where
$$\langle A,M\rangle = \{\langle f,g\rangle \in C(X,Y)\times
C(X,Y):\forall x\in A~\langle f(x),g(x)\rangle\in M\}$$ is called
{\it topology of uniform convergence on elements of $\lambda$} and
denote by $C_{\lambda,\mu}(X,Y)$.
\end{definition}

The well-known fact that if $\lambda$ is a family of all compact
subsets of $X$ or all finite subsets of $X$ then the topology on
$C(X,Y)$ induced by the uniformity $\nu$ of uniform convergence on
elements of $\lambda$ depends only on the topology induced on $Y$
by the uniformity $\mu$ (see \cite{kell, Eng}). In these cases, we
will use the notation $C_c(X,Y)$ and $C_p(X,Y)$, respectively. If
$Y=\mathbb{R}$ then $C_c(X)$ and $C_p(X)$, respectively.

In case, if $(Y,\rho)$ is a metric space and the uniformity $\mu$
is induced by the metric $\rho$, then for $C_{\lambda,\mu}(X,Y)$,
we will use the notation $C_{\lambda,\rho}(X,Y)$ and
$C_{\lambda,\rho}(X)$ for the case $Y=\mathbb{R}$.

If $X\in\lambda$, we write $C_\mu(X,Y)$ in place of
$C_{\lambda,\mu}(X,Y)$ and $C_{\mu}(X)$ in place of
$C_{\mu}(X,\mathbb{R})$.

Recall that a {\it bornology} on a space $X$ is a family $\lambda$
of nonempty subsets of $X$ which is closed under finite unions,
hereditary (i.e. closed under taking nonempty subsets) and forms a
cover of $X$ \cite{HN}. Note that if $\lambda$ is a bornology then
$C_{\lambda,\mu}(X,Y)$ is Hausdorff \cite{Bur75b}.

A base for a bornology $\lambda$ on $X$ is a subfamily $\lambda'$
of $\lambda$ which is cofinal in $\lambda$ with respect to the
inclusion, i.e. for each $A\in \lambda$ there is $B\in \lambda'$
such that $A\subseteq B$. A base is called {\it closed} if all its
members are closed subsets of $X$. Note that if a bornology
$\lambda$ has a closed base, then $A\in\lambda$ implies
$\overline{A}\in \lambda$.

{\it Throughout the paper we suppose that a bornology $\lambda$ on
a space $X$ is a bornology with a closed base.}

\begin{definition}(\cite{Arh89b}) Let $\lambda$ be a family
of subsets of a topological space $X$ and $Y$ be a topological
space. The space $X$ is {\it functionally generated by the family
$\lambda$ with respect to $Y$} if the following condition holds:
for every discontinuous function $f:X\rightarrow Y$ there is an
$A\in \lambda$ such that the function $f|A$ cannot be extended to
a $Y$-valued continuous function on all of $X$.
\end{definition}

Recall that a subset $A$ of a topological space $X$ is called

\medskip

$\bullet$ {\it relatively compact} if $A$ has a compact closure in $X$;

$\bullet$ {\it relatively sequential compact} if every sequence from $A$ has a convergent subsequence whose limit is in $X$;

$\bullet$ {\it sequential compact} if every sequence from $A$ has a convergent subsequence whose limit is in $A$;

$\bullet$  {\it countably compact} if each sequence of $A$ has a
cluster point in $A$;

$\bullet$  {\it relatively countably compact} if each sequence of
$A$ has a cluster point in $X$.

\medskip

For other notation and terminology almost without exceptions we follow the Engelking's book \cite{Eng} and the papers \cite{Os,Os20}.

\section{Dieudonn\'{e} completeness of function spaces}

\begin{theorem}(\cite{Kats}) Let $X$ be a subspace of a Dieudonn\'{e} complete space $Y$ and for any point
$y\in Y\setminus X$ there is a $G_{\delta}$-set $V$ containing $y$
such that $V\subseteq Y\setminus X$. Then $X$ is Dieudonn\'{e}
complete.
\end{theorem}

 Let $X$ be a topological space, $\lambda\subseteq 2^X$. Recall that the space $X$ is called

$\bullet$ a {\it $\lambda_f$-space}, if whenever $f:\, X\to Y$,
where $Y$ is an arbitrary Tychonoff space, is continuous if and
only if  $f|_{\overline{A}}:\, \overline{A}\to Y$ is continuous
for all $A\in\lambda$ (Def. 3.8 in \cite{Os23});

$\bullet$ a {\it $\lambda$-space}, if for any $A\subseteq X$, $A$
is closed in $X$ if and only if $A\cap \overline{B}$ is closed in
$\overline{B}$ for all $B\in \lambda$ (Def. 3.1 in \cite{Os23}).

\medskip

Note that any $\lambda$-space is also $\lambda_f$-space and any
$\lambda_f$-space is functionally generated by $\lambda$ with
respect to any Tychonoff space $Y$.

\begin{lemma} Let $X$ be a topological space, $\lambda$  be a bornology  and $(Y,\rho)$ be a metric space.
If $X$ is a $\lambda_f$-space then $C_{\lambda,\rho}(X,Y)$ is
Dieudonn\'{e} complete.

\end{lemma}

\begin{proof} By Theorem 4.5 in \cite{Os23}, $C_{\lambda, \rho}(X,Y)$ is uniformly homeomorphic to the inverse
limit of the system
$S(X,\lambda,Y)=\{C_\rho(\overline{A},Y),\pi^{A_1}_{A_2},\lambda\}$. Note that for any $A\in \lambda$ the space $C_\rho(\overline{A},Y)$ is metrizable by metric $\rho^*(f,g)=\max \{1, \sup\limits_{x\in \overline{A}}(\rho(f(x),g(x))\}$ where $\rho$ is a metric of $Y$.

Since the inverse
limit of the system
$S(X,\lambda,Y)$ is closed subset of $\prod \{ C_\rho(\overline{A},Y): A\in \lambda\}$, $C_{\lambda,\rho}(X,Y)$ is Dieudonn\'{e} complete.
\end{proof}

\begin{lemma} Let $X$ be a topological space, $\lambda$ be a bornology and $(Y,\rho)$ be a metric space. Then sets

$U_{*}(f,A,\epsilon)=\{g\in C(X,Y): \exists a_g<\epsilon: \forall x\in A$  $\rho(f(x),g(x))\leq a_g\}$,
 $A\in \lambda$ and $\epsilon>0$ are open sets in $C_{\lambda,\rho}(X,Y)$. Moreover, these sets form the base of the space $C_{\lambda,\rho}(X,Y)$.

\end{lemma}

\begin{proof} Let $V_{\epsilon}=\{(x,y)\in Y\times Y: \rho(x,y)<\epsilon\}\subseteq Y\times Y$ be an element of the uniformity of
$Y$ and  $\langle g, A,V_{\epsilon}\rangle:=\{h: \forall x\in A$
$\rho(h(x),g(x))<\epsilon\}$ where $\epsilon>0$.

 Note that for any $g\in U_{*}(f,A,\epsilon)$ there is $\epsilon_1>0$ such that $\langle g, A,V_{\epsilon_1}\rangle \subseteq U_{*}(f,A,\epsilon).$

Indeed, choose $\epsilon_1$ such that   $0<\epsilon_1<\epsilon-a_g$. Then

$\rho(h(x),f(x))\leq
\rho(h(x),g(x))+\rho(g(x),f(x))<\epsilon_1+a_g<\epsilon$ for any
$h\in \langle g, A,V_{\epsilon_1}\rangle$ and $x\in A$, i.e.,
$h\in U_{*}(f,A,\epsilon)$. This shows, that $U_{*}(f,A,\epsilon)$
is open.

Note that $U_{*}(f,A,\epsilon)\subseteq \langle f,
A,V_{\epsilon}\rangle$ for any $f\in C(X,Y)$, $A\in \lambda$ and
$\epsilon>0$. Hence, the set $\{U_{*}(f,A,\epsilon): f\in C(X,Y),
A\in \lambda, \epsilon>0\}$ is a base of $C_{\lambda,\rho}(X,Y)$.

\end{proof}

Let $X$, $Y$ be sets, $Z$ be a topological space and
$f:X\rightarrow Y$ be a map (between {\it sets} $X$ and $Y$).
Define the map $f^\#:Z^Y\rightarrow Z^X$ (between {\it topological
spaces}) dual to $f$ as follows: if $\phi\in Z^Y$ then
$f^\#(\phi)(x)=\phi(f(x))$ for all $x\in X$, i.e.
$f^\#(\phi)=\phi\circ f$. Note that if $f$ is continuous then
$f^\#(C(Y,Z))\subseteq C(X,Z)$.

Let $X$ be a topological space, $\lambda\subseteq 2^X$. Let
$Q=\{B\subseteq X: \overline{A}\cap B$ is closed in $\overline{A}$
for all $A\in \lambda\}$. Let $X_{\lambda}$ be a set $X$ with the
topology $\tau=\{X\setminus B: B\in Q\}$. Note that if $X$ is a
$\lambda$-space, then $X_{\lambda}$ is identical to $X$.

Let $e: X_{\lambda}\rightarrow X$ be an identity map on $X$. Since
the topology on $X$ is weaker than the topology on $X_{\lambda}$,
the map $e$ is continuous. Thus $e$ is a condensation (= a
continuous bijection). Note that $e|_{\overline{A}}:
\overline{A}^{\tau}\rightarrow \overline{A}$ is a homeomorphism
for each $A\in \lambda$.

Put $e^{-1}(\lambda)=\{e^{-1}(A): A\in \lambda\}$. Note that as a
family of sets, it is identical to $\lambda$, also note that
$X_{\lambda}$ is a $e^{-1}(\lambda)$-space (resp.
$\lambda$-space). Further, we will call $X_{\lambda}$ as {\it
$\lambda$-leader} of $X$.

Note that $e^\#$ is an embedding of $C_{\lambda,\rho}(X,Y)$ into
$C_{e^{-1}(\lambda),\rho}(X_{\lambda},Y)$.
\medskip

If $\lambda$ is a family of subsets of a topological space $X$,
then the family of all countable unions of elements of $\lambda$
will be denoted by $\sigma\lambda$.

If $X$ is functionally generated by $\lambda$ it is also
functionally generated by $\sigma\lambda$ (with respect to the
same $Y$). Thus every $\lambda_f$-space is functionally generated
by $\sigma\lambda$ and hence the following theorem is a
generalization of Lemma 3.2.

\begin{theorem}\label{th14} Let $X$ be a topological space, $\lambda$ be a bornology, $(Y,\rho)$ be a metric space. If $X$ is functionally
generated with respect to $Y$ by $\sigma\lambda$, then
$C_{\lambda,\rho}(X,Y)$ is Dieudonn\'{e} complete.

\end{theorem}

\begin{proof} Let $e: X_{\lambda}\rightarrow X$ be  a natural condensation from
$\lambda$-leader $X_\lambda$ of $X$ onto $X$. By Lemma 3.2 we have
that $C_{e^{-1}(\lambda),\rho}(X_{\lambda},Y)$ is Dieudonn\'{e}
complete so using Theorem 3.1 it is enough to prove that
$C_{e^{-1}(\lambda),\rho}(X_{\lambda},Y)\setminus
e^\#(C_{\lambda,\rho}(X,Y))$ is a union of $G_{\delta}$ sets.

Let $f\in C_{e^{-1}(\lambda),\rho}(X_{\lambda},Y)\setminus
e^\#(C_{\lambda,\rho}(X,Y))$. Consider a function $g:X\rightarrow
Y$ such that $g(x)=f(e^{-1}(x))$. Since $X$ is functionally
generated with respect to $Y$ by $\sigma\lambda$, there is
$B=\bigcup\limits_{i=1}^n A_i\in \sigma \lambda$ where
$A_i\in\lambda$ for every $i=1,...,n$, such that $g|_B$ does not
extend to any continuous function. Construct $G_{\delta}$-set $V$
containing $f$:

$V=\bigcap \{U_{*}(f, \bigcup\limits_{i=1}^n e^{-1}(A_i),
\epsilon_n): n\in \mathbb{N}\}$ where $\epsilon_n=\frac{1}{n}$.
Note that $V\cap e^\#(C_{\lambda,\rho}(X,Y))=\emptyset$. Indeed,
for any function $h\in V$ the function $h\circ e^{-1}|_B=g|_B$
and, hence, $h\circ e^{-1}\notin C_{\lambda,\rho}(X,Y)$.

\end{proof}

Since any Dieudonn\'{e} complete space is a $\mu$-space, we get the following corollary.

\begin{corollary} {\it Let $X$ be a topological space, $\lambda$ be a bornology, $(Y,\rho)$ be a metric space. If $X$ is functionally generated with respect to $Y$ by
$\sigma\lambda$ then $C_{\lambda,\rho}(X,Y)$ is a $\mu$-space.}

\end{corollary}

Let's try to reverse the previous Theorem \ref{th14}.

Let us remind a modification of an arbitrary topological space
$X$ and an arbitrary family $\lambda\subseteq 2^X$ to a space
$X_{\lambda}$ consisting of the same set of points \cite{Os23}.

\begin{proposition}\label{pr16} Let $X$ be a topological space, $\lambda$ be a bornology, $(Y,\rho)$ be a metric space and
$e: X_{\lambda}\rightarrow X$ be a natural condensation from
$\lambda$-leader $X_\lambda$ of $X$  onto $X$. If for any $f\in
C_{e^{-1}(\lambda),\rho}(X_{\lambda},Y)\setminus
e^\#(C_{\lambda,\rho}(X,Y))$ there is a $G_{\delta}$-set $V$
containing $f$ such that $V\cap
e^\#(C_{\lambda,\rho}(X,Y))=\emptyset$ then $X$ is functionally
generated with respect to $Y$ by $\sigma\lambda$.

\end{proposition}

\begin{proof} Let $f\in C_{e^{-1}(\lambda),\rho}(X_{\lambda},Y)\setminus e^\#(C_{\lambda,\rho}(X,Y))$
and $V$ is a $G_{\delta}$-set containing $f$ such that $V\cap
e^\#(C_{\lambda,\rho}(X,Y))=\emptyset$. Since $\{U_*(f,
e^{-1}(A),\epsilon): f\in C(X_{\lambda},Y), A\in\lambda,
\epsilon>0\}$ is a base of
$C_{e^{-1}(\lambda),\rho}(X_{\lambda},Y)$, we can assume that
$V=\bigcap \{U_{*}(f, \bigcup\limits_{i=1}^n e^{-1}(A_i),
\varepsilon_n): n\in \mathbb{N}\}$  where $A_i\in \lambda$ for
each $i=1,...,n$ and $\varepsilon_n\leq \frac{1}{n}$ for every
$n\in \mathbb{N}$. Let $B\in \sigma\lambda$, $W=\{g=h\circ e^{-1}:
h\in V\}$ and $F=f\circ e^{-1}$. Note that $g|_B=F|_B$ for any
$g\in W$. Since $V\cap e^\#(C_{\lambda,\rho}(X,Y))=\emptyset$, any
$g\in W$ is discontinuous. Thus, we proved that $F|_B$ does not
extend to any continuous function for any $B\in \sigma\lambda$.
Let $f\in Y^X\setminus C(X_{\lambda},Y)$. Then there is $A\in
\lambda$ such that $f|_A$ is discontinuous ($X_{\lambda}$ is a
$\lambda_f$-space). Hence, $X$ is functionally generated with
respect to $Y$ by the family $\sigma\lambda$.
\end{proof}

Note that Proposition \ref{pr16} is not a complete inversion of Theorem \ref{th14}. Moreover, as Theorem \ref{th39} will show,
Theorem \ref{th14} is not completely invertible.

\medskip

A cardinal $\tau$ is called {\it Ulam measurable} if on the set of cardinality $\tau$ there is a maximal centered system with empty intersection, such that the intersection of any countable subfamily of it is not empty (i.e. countably centered).

In consistency with the axioms of set theory of the system $ZFC$,
we may assume that measurable cardinals do not exist. Under this
assumption every Dieudonn\'{e} complete space is realcompact. The
converse is always true.

\begin{theorem} Let $X$ be a topological space, $\lambda$ be a
bornology and $(Y,\rho)$ be a metric space. If $X$ is functionally
generated with respect to $Y$ by $\sigma\lambda$ and
$c(C_{\lambda,\rho}(X,Y))$ is not Ulam measurable cardinal then
$C_{\lambda,\rho}(X,Y)$ is realcompact.
\end{theorem}

\begin{proof} By result in \cite{Shir}, if a Tychonoff space $Z$ such that $c(Z)$ is not Ulam measurable cardinal then Dieudonn\'{e} completeness of $Z$ implies realcompactness of $Z$.
It remains to apply Theorem \ref{th14}.
\end{proof}

\begin{corollary} {\it Assume that Ulam measurable cardinals do not exist. Let $X$ be a topological space, $\lambda$ be a
bornology and $(Y,\rho)$ be a metric space. If $X$ is functionally
generated with respect to $Y$ by $\sigma\lambda$ then
$C_{\lambda,\rho}(X,Y)$ is realcompact.}

\end{corollary}

Next we give an example showing that Theorem \ref{th14} cannot be reversed.

\begin{theorem}\label{th39} There exists Tychonoff space $X$ such that $C_{\lambda,\rho}(X)$ is
realcompact (a fortiori it is Dieudonn\'{e} complete) where
$\lambda$ is a bornology with a base $\lambda'$ where $\lambda'$
is the family of all countable compact subsets of $X$ and  $X$ is
not functionally generated with respect to $\mathbb{R}$ by the
family $\sigma\lambda$.
\end{theorem}

\begin{proof} Let $X$ be an ordinal $\omega_1+1$ equipped with the order topology.

Since $\lambda'$ is a family of all metrizable compact subsets of
$X$, the space $C_{\lambda,\rho}(X)$ is  a (hereditary)
paracompact space (see Theorem 6 in \cite{Gul}), hence, it is
realcompact.

\medskip
The function

$$ f(\alpha)= \left\{
\begin{array}{lcr}
0, \ \ \ \ if \, \alpha <\omega_1, \\
1, \, \, \, \, if \, \, \, \, \alpha=\omega_1\\
\end{array}
\right.
$$

is discontinuous, but since every $S\in \sigma\lambda$ is
countable $f\upharpoonright S$ is continuous.

Note that for any countable set $S\in \sigma\lambda$ there is
$h\in C(X)$ such that $h\upharpoonright S=f\upharpoonright S$.

If the point $\omega_1\not\in S$ then there is $\beta<\omega_1$
such that $S\subset \{\xi: \xi<\beta\}$. Let $h(\alpha)=0$ for any
$\alpha\in X$.

If the point $\omega_1\in S$ then there is $\beta<\omega_1$ such
that $S\setminus \{\omega_1\}\subset \{\xi: \xi<\beta\}$. Let

$$ h(\alpha)= \left\{
\begin{array}{lcr}
0, \ \ \ \ if \, \alpha <\beta^+, \\
1, \, \, \, \, if \, \, \, \, \alpha\geq \beta^+\\
\end{array}
\right.
$$

Thus, $X$ is not functionally generated with respect to
$\mathbb{R}$ by the family $\sigma\lambda$.

\end{proof}

\section{Eberlein--\v{S}mulian theorem}

\medskip

\begin{definition}
 A Hausdorff topological space $E$ is called {\it angelic} if every relatively countably compact set $K$
 in $E$ is relatively compact and for every $x\in \overline{K}\setminus K$ there exists a sequence in $K$ converging to $x$.
\end{definition}

In angelic spaces, the (relatively) countably compact, (relatively) compact and (relatively) sequentially compact sets are the same; see (\cite{freml}, Theorem 3.3), and also
\cite{Gov}, where a proper subclass of angelic spaces has been studied.

Classical examples of angelic spaces include spaces $C_p(X)$ with
compact $X$ (see, for example, \cite{190, 240}), Banach spaces $E$
with the weak topology $\sigma(E,E')$ \cite{pric} and spaces of
first Baire class functions on any Polish space $P$ with the
topology of the pointwise topology on $P$ (see \cite{Bur}).

\begin{lemma}(\cite{57}) Let $f:Y\rightarrow X$ be a condensation from a Tychonoff space $Y$ onto an angelic Tychonoff space $X$. Then $Y$ is angelic.

\end{lemma}

\begin{theorem}(\cite{freml}, Theorem 3.5)
If $C_p(X)$ is angelic, then $C_p(X,Z)$ is angelic for any metric
space $Z$.
\end{theorem}

In \cite{57} and \cite{48}, it is proved some generalization of the Eberlein--\v{S}mulian theorem for angelic spaces.

\begin{theorem} \label{th2} If $A$ is a subset of an angelic space $X$, then the following statements are equivalent:

(1) $A$ is relatively compact;

(2) $A$ is relatively sequential compact;

(3) $A$ is relatively countably compact.

\end{theorem}

In \cite{AsVel}, classes of the space $X$ were found for which
relatively compact sets in $C_p(X)$ are defined as bounded or
countably compact. This allows us to continue generalizing the
Eberlein--\v{S}mulian theorem for the class of Banach spaces.

\medskip
Let us recall some definitions that will be needed in the
following theorem. Let $X$ be a normed linear space and $X'$ be
the continuous dual of $X$. The {\it weak topology} on $X$,
denoted by $\sigma(X,X')$, is defined to be the coarsest topology
such that each $x'\in X'$ is continuous.

Recall that  a locally convex topology $\tau$ on a normed linear
space $(X,\|\cdot\|)$ is called {\it compatible} for the pair
$(X,X')$ if $(X,\tau)'=X'$ as vector spaces. Note that in this
case $\tau\supseteq \sigma(X,X')$ (Ch.IV in \cite{50}).

\begin{theorem} Let $(X,\|\cdot\|)$ be a Banach space,  $\tau$ be a locally convex topology on $X$ compatible for the pair $(X,X')$ and $A\subseteq X$. Then
the following statements are equivalent for the space $(X,\tau)$:

(1) $A$ is relatively compact;

(2) $A$ is relatively sequential compact;

(3) $A$ is relatively countably compact;

(4) $A$ is (topologically) bounded.

\end{theorem}

\begin{proof} Since the topology $\tau$ on $X$ is finer than the weak topology
$\sigma(X,X')$, the identity map $e: (X,\tau)\rightarrow (X,
\sigma(X,X'))$ is a condensation. By Lemma 4.2, $(X,\tau)$ is an
angelic space. Thus statements (1), (2) and (3) are equivalent by
Theorem 4.4. Given that the implication $(1)\Rightarrow(4)$ is
obvious, it is sufficient to prove $(4)\Rightarrow(1)$.

Since $\tau$ is a locally convex topology on $X$ compatible for
the pair $(X,X')$, $\tau$ is a topology of uniform convergence on
elements of $\lambda$ for a bornology $\lambda$ on $X'$ where
$\lambda$ consisting of relatively compact subsets of $(X',
\sigma(X',X))$ (see Theorem 3.2, Ch.IV in \cite{50}). Thus, the
space $(X,\tau)$ can be considered as a subspace of space $X''$
with a topology induced from the space $C_{\lambda,\rho}(X')$
where $X'$ is the continuous dual space to $X$, $X''$ is the
continuous double dual space to $(X,\tau)$. Then, we can consider
the space $(X,\tau)$ as subspace of
$C_{e^{-1}(\lambda),\rho}(X'_{\lambda})$ where $e:
X'_{\lambda}\rightarrow X'$ is a natural condensation from
$\lambda$-leader $X'_{\lambda}$ of $X'$ onto $X'$ \cite{Os23}.

We denote by $X_{\sigma}$ the space $(X,\sigma(X,X'))$. The space
$X_{\sigma}$ is homeomorphically embedded to $C_p(X')$ and
$C_p(X')$ is homeomorphically embedded to $C_p(X'_{\lambda})$. The
condensation $id:
C_{e^{-1}(\lambda),\rho}(X'_{\lambda})\rightarrow
C_p(X'_{\lambda})$ (identity mapping) induces a condensation
$id_X: (X,\tau)\rightarrow X_{\sigma}$.

Let $A$ be a bounded subset of $(X,\tau)$. Since $X'_{\lambda}$ is
a $\lambda$-space (and, hence, $\lambda_f$-space), by Proposition
4.7 in \cite{Os23}, the space
$C_{e^{-1}(\lambda),\rho}(X'_{\lambda})$ is a complete uniform
space. Thus, $C_{e^{-1}(\lambda),\rho}(X'_{\lambda})$ is a
Dieudonn\'{e} complete space and, hence, it is a $\mu$-space.
Since $A$ is a bounded subset of $(X,\tau)$ and $(X,\tau)$ is a
subset of $C_{e^{-1}(\lambda),\rho}(X'_{\lambda})$, $A$ is a
bounded subset of $C_{e^{-1}(\lambda),\rho}(X'_{\lambda})$. It
follows that
$\overline{A}^{C_{e^{-1}(\lambda),\rho}(X'_{\lambda})}$ is
compact.

  On the other hand, $id_X(A)$ is bounded in $X_{\sigma}$. Note that
$id_X(\overline{A}^X)\subseteq \overline{id_X(A)}^{X_{\sigma}}$.

Recall that $X':=(X,\tau)'=(X,\|\cdot\|)'$. Thus, $X'$ is a Banach
space. Since $X'$ is first countable, by Theorem 1 in
\cite{AsVel}, $\overline{id_X(A)}^{C_p(X')}$ is compact.

Since $X_{\sigma}$ is homeomorphically embedded to $C_p(X')$ and
the homeomorphic image of $X_{\sigma}$ is closed in $C_p(X')$,
$\overline{id_X(A)}^{X_{\sigma}}=\overline{id(A)}^{C_p(X')}=\overline{id(A)}^{C_p(X'_{\lambda})}$.
Since the mapping $id$ is a condensation and
$\overline{A}^{C_{e^{-1}(\lambda),\rho}(X'_{\lambda})}$ is
compact,
$id(\overline{A}^{C_{e^{-1}(\lambda),\rho}(X'_{\lambda})})=\overline{id
(A)}^{C_p(X'_{\lambda})}$. Note that $id\upharpoonright
(\overline{A}^{C_{e^{-1}(\lambda),\rho}(X'_{\lambda})}):
\overline{A}^{C_{e^{-1}(\lambda),\rho}(X'_{\lambda})} \rightarrow
\overline{id (A)}^{C_p(X'_{\lambda})}$ is a homeomorphism.

Since $id_X(\overline{A}^X)\subseteq
\overline{id_X(A)}^{X_{\sigma}}$ and
$\overline{id_X(A)}^{X_{\sigma}}\subseteq
\overline{id(A)}^{C_p(X'_{\lambda})}$,
$\overline{A}^X=\overline{A}^{C_{e^{-1}(\lambda),\rho}(X'_{\lambda})}$.

It follows that $\overline{A}^X$ is compact.
\end{proof}

\begin{corollary} {\it  Let $(X,\|\cdot\|)$ be a Banach space and $\tau$ be a locally convex topology on $X$ compatible for the pair $(X,X')$. Then $(X,\tau)$ is a $\mu$-space.}

\end{corollary}



{\bf Acknowledgements.} The authors would like to thank the
referee for careful reading and valuable comments.

\bibliographystyle{model1a-num-names}
\bibliography{<your-bib-database>}

\end{document}